\documentclass[reqno, a4paper, 11pt]{amsart}
\usepackage{amssymb,amsmath,mathtools,amsthm}
\usepackage{enumerate}
\usepackage{xcolor}
\usepackage{mathrsfs}
\usepackage{enumitem}

\makeatletter
\newcommand{\mylabel}[2]{#2\def\@currentlabel{#2}\label{#1}}
\makeatother

\newcommand{\F}{\mathcal F}
\newcommand{\AW}{\mathcal{AW}}
\newcommand{\Law}{\mathscr L}
\newcommand{\FP}{\mathrm{FP}}
\newcommand{\FFP}{\textbf{FP}}
\newcommand{\pp}{\mathrm{pp}}
\newcommand{\fp}[1]{{\bf #1}}
\newcommand{\W}{\mathcal W}
\newcommand{\N}{\mathbb N}
\newcommand{\R}{\mathbb R}
\newcommand{\X}{\mathcal X}

\newcommand{\Pc}{\mathcal P}
\newcommand{\proj}{\mathrm{pj}}
\newcommand{\cpl}{\mathrm{Cpl}}
\newcommand{\cpla}{\cpl_{\mathrm{c}}}
\newcommand{\cplba}{\cpl_{\mathrm{bc}}}

\newcommand{\CW}{\mathcal{CW}}
\newcommand{\SCW}{\mathcal{SCW}}

\def\P{{\mathbb P}}
\def\E{{\mathbb E}}

\def\Q{{\mathbb Q}}
\def\R{{\mathbb R}}

\usepackage{color}

\usepackage{hyperref}
\usepackage{txfonts}

\usepackage{newtxtext}

\usepackage{tikz}
\usetikzlibrary{external}
\usetikzlibrary{backgrounds,patterns,decorations.pathreplacing,decorations.pathmorphing}
\usepackage{caption}
\usepackage{accents}

\newtheorem{theorem}{Theorem}

\newtheorem{corollary}[theorem]{Corollary}
\newtheorem{definition}[theorem]{Definition}
\newtheorem{lemma}[theorem]{Lemma}
\newtheorem{proposition}[theorem]{Proposition}

\theoremstyle{remark}
\newtheorem{remark}[theorem]{Remark}

\begin{document}

\begin{abstract}
	The adapted weak topology is an extension of the weak topology for stochastic processes designed to adequately capture properties of underlying filtrations.
	With the recent work of Bart--Beiglb\"ock--P.\ \cite{BaBePa21} as starting point, the purpose of this note is to recover with topological arguments the intriguing result by Backhoff--Bartl--Beiglb\"ock--Eder \cite{BaBaBeEd19b} that all adapted topologies in discrete time coincide.
	We also derive new characterizations of this topology including descriptions of its trace on the sets of Markov processes and processes equipped with their natural filtration.
	To emphasize the generality of the argument, we also describe the classical weak topology for measures on $\mathbb R^d$ by a weak Wasserstein metric based on the theory of weak optimal transport initiated by Gozlan--Roberto--Samson--Tetali \cite{GoRoSaTe17}.
	\medskip

	\noindent\emph{keywords:} adapted weak topology; stochastic processes
\end{abstract}

\author{Gudmund~Pammer}
\thanks{ETH Zurich, Switzerland, \href{mailto:gudmund.pammer@math.ethz.ch}{gudmund.pammer@math.ethz.ch}}
\title{A note on the adapted weak topology in discrete time}

\maketitle

\section{Introduction}
\label{sec:introduction}

An essential difference in the study of random variables and stochastic processes is that the latter comes in conjuction with filtrations that are designed to model the flow of available information:
Let us consider a path space $\mathcal X := \prod_{t = 1}^N \mathcal X_t$ equipped with the product topology where $(\mathcal X_t, d_{\mathcal X_t})$ are Polish metric spaces and $N \in \mathbb N$ denotes the number of time steps.
We write $\mathcal P(\mathcal X)$ for the set of laws of stochastic processes, i.e., Borel probability measures on $\mathcal X$.
Canonically, we identify the law $\mathbb P \in \Pc(\X)$ with the process 
\begin{equation}
	\label{eq:def_plain_process}
	\big( \X, \sigma(X_{1:t})_{t = 1}^N, \sigma(X), \P, X \big),
\end{equation}
where $X = X_{1:N}$ is the coordinate process on $\X$, $X_{1:t}$ denotes the projection from $\X \to \prod_{t = 1}^t \X_s =: \X_{1:t}$, and $\sigma(X_{1:t})$ the $\sigma$-algebra generated by $X_{1:t}$.
For $\P, \Q \in \mathcal P_p(\X)$, that are probabilities in $\Pc(\X)$ with finite $p$-th moment, $p \in [1,\infty)$, the $p$-Wasserstein distance $\W_p$ is given by
\begin{equation}
	\label{eq:def_Wasserstein_distance}
	\W_p^p(\P, \Q) := \inf_{\pi \in \cpl(\P,\Q)} \E_\pi \big[ d_\X^p(X, Y) \big],
\end{equation}
where $\cpl(\P,\Q)$ denotes the probabilities on $\X \times \X$ with marginals $\P$ and $\Q$, and $d_\X^p(x,y) := \sum_{t = 1}^N d_{\X_t}^p(x_t,y_t)$.
We equip $\Pc_p(\X)$ with the topology induced by $\W_p$ and note that if $d_\X$ is bounded, $\W_p$ metrizes the weak topology on $\Pc(\X)$.

The starting point for the study of \emph{adapted} topologies poses the fact that probabilistic operations and optimization problems that crucially depend on filtrations, such as the Doob decomposition, the Snell envelope, optimal stopping, utility maximization, and stochastic programming, are typically not continuous w.r.t.\ weak topologies.
These shortcomings are acknowledged by several authors from different communities, see e.g.\ \cite{Al81, HoKe84, PfPi12, BaBaBeEd19a,BaBePa21} for more details.
The purpose of this note is to recover and strengthen the main result of Backhoff et al.\ \cite{BaBaBeEd19b} that all adapted topologies on $\Pc(\X)$ coincide.
In comparison to the original proof, our argument is more conceptional: at its core lies the elementary fact that comparable compact Hausdorff topologies agree.

\subsection{Stochastic processes and the adapted weak topology}

Subsequently, we want to consider topologies that incorporate the flow of information encoded in filtrations, for processes on general filtered probability spaces.
Therefore, we follow the approach of \cite{BaBePa21} by introducing the notion of a \emph{filtered process}.
\begin{definition}[Filtered process]
	A filtered process $\fp X$ with paths in $\X$ is a $5$-tuplet
	\begin{equation}
		\label{eq:def_filtered_process}
		\big( \Omega^\fp X, (\F_t^\fp X)_{t = 1}^N, \F^\fp X, \P^\fp X, X \big),
	\end{equation}
	consisting of a complete filtered probability space $(\Omega^\fp X, (\F_t^\fp X)_{t = 1}^N, \F^\fp X, \P^\fp X)$ and an $(\F_t^\fp X)_{t = 1}^N$-adapted stochastic process $X$ with paths in $\X$.
	We write $\FP$ for the class of all filtered processes with paths in $\X$, and $\FP_p$ for the subclass of filtered processes that finitely integrate $d_\X^p(\hat x, X)$ for some $\hat x \in \X$.
\end{definition}

Although, a-priori $\FP$ is a proper class (that contains a lot of redundancy), in the following we will consider equivalence classes $[\fp X]$ of filtered processes in the sense of Hoover-Keisler \cite{HoKe84} such that the corresponding factor space $\FFP$ becomes a set, see for example \cite{BaBePa19}.
This factorization can be seen similarly as in classical $L^p$-theory where one considers equivalence classes modulo almost-sure equivalence in order to obtain a Banach space.
This equivalence relation can be characterized by an \emph{adapted} version of the Wasserstein distance, c.f.\ \cite[Theorem 1.5]{BaBePa21}, the \emph{adapted Wasserstein distance} $\AW_p$ which will be introduced in detail in Section \refeq{ssec:adapted_topologies} below: for $\fp X, \fp Y \in \FP_p$ we have
\[ \fp Y \in [\fp X] \iff \AW_p(\fp X, \fp Y) = 0. \]

Henceforth, we consider the factor space $\FFP$ and remark that equivalent processes share the same probabilistic properties, e.g.\ being adapted, having the same Doob decomposition and Snell-envelope, \dots.
Moreover, we write $\FFP_p$ for those elements $\fp X \in \FFP$ with $\E_{\P^\fp X}[d_\X^p(\hat x, X)] < \infty$ for some $\hat x \in \X$.

The topology induced by the adapted Wasserstein distance is denoted by $\tau_\AW$ and called the \emph{adapted weak topology}.
When equipping $\FFP$ with the adapted weak topology, we obtain a space rich of topological and geometric properties, see \cite{BaBePa21}.
Importantly, we note that as a consequence of the adapted block approximation introduced in \cite{BaBePa21} the values of $\AW_p(\fp X, \fp Y)$ (and also $\CW_p(\fp X, \fp Y)$ which will be introduced down below) is independent of the particular choice of representatives.
Similarly, we can equip $\FFP_p$ with $p$-th Wasserstein topology by letting
\[
	\W_p(\fp X, \fp Y) := \W_p(\Law(X),\Law(Y)),
\]
and remark that $\W_p$ is not point seperating on $\FFP_p$.
Processes can have the same law, but very different information structure, see for instance \cite[Figure 1]{BaBaBeEd19a}.
An important feature of $\AW_p$ is the following Prokhorov-type result which will be applied at several occasions in the proofs:

\begin{theorem}[Theorem 1.7 of \cite{BaBePa21}]
	\label{thm:precompact}
	A set $M \subset \FFP_p$ is $(\FFP_p,\AW_p)$-precompact if and only if $M$ is $(\FFP_p,\W_p)$-precompact, that is $\{ \Law(X) \colon \fp X \in M \} \subset \Pc_p(\X)$ is precompact.
\end{theorem}

To emphasize the significance of Theorem \refeq{thm:precompact} and to give the idea behind the main results, we formulate the following immediate corollary:

\begin{corollary}
	\label{cor:equivalent_topologies}
	Let $d \colon \FFP_p \times \FFP_p \to \R^+$ be a metric on $\FFP_p$ such that
	\begin{equation}
		\label{eq:d_comparable}
		\W_p(\fp X, \fp Y) \le d(\fp X, \fp Y) \le \AW_p(\fp X, \fp Y).
	\end{equation}
	Then $d$ metrizes the adapted weak topology $\tau_\AW$.
\end{corollary}

Indeed, by \eqref{eq:d_comparable} we find
\begin{equation}
	\label{eq:AW_implies_d}
	\lim_{k \to \infty} \AW_p(\fp X^k, \fp X) \implies \lim_{k\to \infty} d(\fp X^k, \fp X).
\end{equation}
On the other hand, to deduce the reverse implication of \eqref{eq:AW_implies_d}, let $(\fp X^k)_{k \in \N}$ be a $d$-convergent sequence with limit $\fp X$.
Then the sequence is $\W_p$-precompact and therefore  $\AW_p$-precompact by Theorem \refeq{thm:precompact}.
Therefore, there exist $\fp Y \in \FFP_p$ and a subsequence with $
\lim_{j \to \infty} \AW_p(\fp X^{k_j}, \fp Y) = 0$.
By \eqref{eq:AW_implies_d} this sequence also converges w.r.t.\ $d$, thus, the triangle inequality yields $d(\fp X, \fp Y) = 0$.
Finally, as $d$ is a metric, we get that $\fp X = \fp Y$ and thus $\lim_{k \to \infty} \AW_p(\fp X^k, \fp X) = 0$.

\subsection{Adapted topologies}
\label{ssec:adapted_topologies}

In order to capture the properties of filtrations, numerous authors have introduced extensions of the weak topology of measures on $\mathcal P(\mathcal X)$, which we frame in our setting and briefly introduce below. For a thorough overview of the topic and introduction to those topologies we refer to \cite{BaBaBeEd19b} and the references therein.

\begin{enumerate}[label=(\arabic*)]
	\item[\mylabel{it:Aldous}{(A)}]
			Aldous \cite{Al81} introduces the \emph{extended weak topology} $\tau_\textrm{A}$ by associating a process $\fp X \in \FFP$ with a measure-valued martingale $\pp^1(\fp X)$, the so-called \emph{prediction process}, that is here
			\begin{equation}
				\label{eq:def_Aldous_prediction_process}
				\pp^1(\fp X) := \big( \Law( X | \F_t^\fp X) \big)_{t = 1}^N \in \Pc(\X)^N,
			\end{equation}
			where $\Law( X | \F_t^\fp X)$ is the conditional law of $X$ given $\F_t^\fp X$.
			Then $\tau_\textrm{A}$ is defined as the initial topology induced by $\fp X \mapsto \Law(\pp^1(\fp X))$ when $\Pc(\Pc(\X)^N)$ is equipped with the weak topology.
	\item[\mylabel{it:Hoover-Keisler}{(HK)}]
			Hoover-Keisler \cite{HoKe84} introduce an increasing sequence of topologies $\tau^r_{\textrm{HK}}$ on $\FFP$ where $r \in \N \cup \{0, \infty \}$ is called the rank.
			This is achieved by iterating Aldous' construction of the prediction process.
			Set $\pp^0(\fp X) := X$ and, recursively define, for $r \in \N \cup \{ \infty \}$,
			\begin{equation}
				\label{eq:def_Hoover-Keisler_prediction_process}
				\pp^r(\fp X) := \big( \Law( \pp^{r - 1} | \F_t ) \big)_{t = 1}^N,
			\end{equation}
			and $\pp(\fp X) := \pp^\infty(\fp X)$.
			Analogously to \ref{it:Aldous}, for $r \in \N \cup \{ 0 , \infty \}$, $\tau_{\textrm{HK}}^r$ is given by the initial topology w.r.t.\ $\fp X \mapsto \Law( (\pp^k(\fp X))_{k = 0}^r )$.
			We remark that $\tau_{\textrm{HK}}^0$ is equivalent to weak convergence of the law, $\tau_{\textrm{HK}}^1 = \tau_{\textrm{A}}$, and $\tau_{\textrm{HK}}^{N-1} = \tau_{\textrm{HK}}^r$ for $r \ge N$ (see \cite{BaBePa21}) and simply write then $\tau_\textrm{HK} := \tau_\textrm{HK}^{N-1}$. 
			% The topology $\tau_\textrm{HK}^{N-1}$ is called the adapted weak topology and will be denoted by $\tau_\textrm{HK}$.
			
	\item[\mylabel{it:Optimal_stopping}{(OS)}]
			The optimal stopping topology $\tau_{\textrm{OS}}$ is defined in \cite{BaBaBeEd19b} as the initial topology w.r.t.\ the family of maps
			\begin{equation}
				\fp X \mapsto \inf \big\{ \E_{\P^\fp X}[c(\rho, X)] \colon \rho \text{ is } (\F_t^\fp X)_{t=1}^N\text{-stopping time} \big\},
			\end{equation}
			where $c \colon \{1,\ldots,N\} \times \X \to \R$ is continuous, bounded, and non-anticipative, that is $c(t,x) = c(t,y)$ if $x_{1:t} = y_{1:t}$ for $(t,x), (t,y) \in \{1,\ldots,N\} \times \X$.
	\item[\mylabel{it:Hellwig}{(H)}]
		The information topology $\tau_{\textrm{H}}$ of Hellwig \cite{He96} is based on a similar point of view as \ref{it:Aldous} and \ref{it:Hoover-Keisler}.
		Properties of the filtration are encoded in the laws
		\begin{equation}
			\label{eq:def_Hellwig}
			\Law\big( X_{1:t},\Law\big( X_{t + 1:N} | \F_t^\fp X \big) \big), \quad 1 \le t \le N,
		\end{equation}
		that are measures on $\Pc(\X_{1:t} \times \Pc(\X_{t + 1:N}))$.
	\item[\mylabel{it:Bonnier-Liu-Oberhauser}{(BLO)}]
		Let the path space $\X$ be the $N$-fold product of a separable Banach space $V$, i.e., $\X = V^N$.
		In this setting, Bonnier-Liu-Oberhauser \cite{BoLiOb21} embed $\FFP$ into graded linear spaces $V_r$ via higher rank expected signatures, where $r \in \N \cup \{ 0, \infty \}$ is again the rank, and define $\tau_{\textrm{BLO}}^r$ as the initial topology w.r.t.\ the corresponding embedding $\Phi_r \colon \FFP \to V_r$.
\end{enumerate}

\begin{remark}
	In case that $d_\X$ is an unbounded metric on $\X$, we will fix for the rest of the paper $p \in [1,\infty)$ and consider the subset $\FFP_p$ with the following topological adaptation.
	The topologies \ref{it:Aldous}, \ref{it:Hoover-Keisler}, \ref{it:Optimal_stopping}, \ref{it:Hellwig} and \ref{it:Bonnier-Liu-Oberhauser} are then refined by additionally requiring continuity of
	\begin{equation}
		\label{eq:def_moment_convergence}
		\FFP_p \ni \fp X \mapsto \E_{\P^\fp X} \big[ d_\X^p(\hat x, X)\big].
	\end{equation}
	To avoid notational excess, we state all results on $\FFP_p$ for some $p \in [1,\infty)$.
	All results are also true when replacing $\FFP_p$ with $\FFP$ (and if necessary $d_\X$ with, for example, $d_\X \wedge 1$). 
\end{remark}

Besides using the powerful concept of initial topologies, various authors have constructed adapted topologies based on ideas from optimal transportation.
The essence of this approach is to encode filtrations into constraints for the set of couplings and thereby construct modifications of the Wasserstein distance suitable for processes.
To illustrate the idea, recall that optimal transport has so-called transport maps $T \colon \X \to \X$ at its core, satisfying the push-forward condition $T_\# \P = \Q$ for $\P,\Q\in\Pc(\X)$.
We refer to \cite{Vi09} for a comprehensive overview on optimal transport.
In our context, where $\P$ and $\Q$ are laws of processes, \emph{causal optimal transport} suggests to use \emph{adapted} maps in order to transport $\P$ to $\Q$, i.e., $T_\# \P = \Q$ and $T$ is non-anticipative, which means
\[
	T(X) = \big(T_1(X_1), T_2(X_{1:2}),\ldots, T_N(X)\big).
\]
When $X$ resp.\ $Y$ denote the first resp.\ second coordinate projection from $\X \times \X \to \X$, then this additional adaptedness constraint on couplings can be formulated as
\begin{equation}
	\label{eq:def_causal_couplings}
	\cpla(\P,\Q) := \big\{ \pi \in \cpl(\P,\Q) \colon X \perp_{X_{1:t}} Y_{1:t} \text{ under $\pi$ for } t = 1,\ldots,N-1  \big\},
\end{equation}
where, for $\sigma$-algebras $\mathcal A, \mathcal B, \mathcal C$ on some probability space, $\mathcal A \perp_\mathcal B \mathcal C$ denotes conditional independence of $\mathcal A$ and $\mathcal C$ given $\mathcal B$.
Elements of $\cpla(\P,\Q)$ are called causal couplings.
When one symmetrices \eqref{eq:def_causal_couplings} one obtains the set of \emph{bicausal couplings} $\cplba(\P,\Q)$, that are $\pi \in \cpla(\P,\Q)$ such that $(Y,X)_\# \pi \in \cpla(\Q,\P)$.
These definitions can be easily extended to $\FFP$.
\begin{definition}[Causal and bicausal couplings]
	\label{def:causality}
	Let $\fp X, \fp Y \in \FFP$.
	For $s,t \in \{0,\ldots,N\}$ we denote by $\F_{s,t}^{\fp X, \fp Y}$
	the $\sigma$-algebra on $\Omega^\fp X \times \Omega^\fp Y$ given by $\F_s^\fp X \otimes \F_t^\fp Y$ under the convention that $\F_0^\fp X$ and $\F_0^\fp Y$ are the corresponding trivial $\sigma$-algebras.
	A probability $\pi$ on the measurable space $(\Omega^\fp X \times \Omega^\fp Y, \F^\fp X \otimes \F^\fp Y)$ is called causal if, under $\pi$,
	\begin{equation}
		\label{eq:def_causal_FP}
		\F^{\fp X,\fp Y}_{N,0} \perp_{\F_{t,0}^{\fp X,\fp Y}} \F^{\fp X,\fp Y}_{0,t}.
	\end{equation}
	We call $\pi$ bicausal if it additionally satisfies
	\begin{equation}
		\label{eq:def_anticausal_FP}
		\F^{\fp X,\fp Y}_{0,N} \perp_{\F_{0,t}^{\fp X,\fp Y}} \F^{\fp X,\fp Y}_{t,0}.
	\end{equation}
	Finally, we write $\cpla(\fp X, \fp Y)$ resp.\ $\cplba(\fp X,\fp Y)$ for the set of causal resp.\ bicausal probabilities with first marginal $\P^\fp X$ and second marginal $\P^\fp Y$.
\end{definition}

\begin{enumerate}
	\item[\mylabel{it:SCW}{(SCW)}]
		Lassalle \cite{La13} and Backhoff et al.\ \cite{BaBeLiZa16} coin the notion of causality, see Definition \ref{def:causality}, and introduce the causal Wasserstein ``distance'' $\CW_p$ on $\Pc_p(\X)$.
		For $\fp X,\fp Y \in \FFP_p$ we have
		\begin{equation}
			\label{eq:def_CW}
			\CW_p^p(\fp X, \fp Y) :=
			\inf_{\pi \in \cpla(\fp X, \fp Y)}
			\E_{\pi} \big[ d_\X^p(X,Y) \big].
		\end{equation}
		Clearly, $\CW_p$ is not a metric as it lacks symmetry, which motivates to consider the so-called \emph{symmetrized causal Wasserstein distance}, see \cite{BaBaBeEd19b},
		\begin{equation}
			\label{eq:def_symmetrized_CW}
			\SCW_p(\fp X, \fp Y) := \max \big\{ \CW_p(\fp X, \fp Y), \CW_p(\fp Y, \fp X)\big\},
		\end{equation}
		which constitutes a metric on $\FFP_p$. 
		We write $\tau_{\SCW}$ for the induced topology.
	\item[\mylabel{it:AW}{(AW)}]
		Instead of symmetrizing as in \eqref{eq:def_symmetrized_CW}, one can directly symmetrize the definition on the level of couplings via the notion of bicausal couplings.
		Approaches in this spirit but to different extents go back to R\"uschendorf \cite{Ru91}, Pflug-Pichler \cite{PfPi12}, Bion-Nadal-Talay \cite{BiTa19}, and Bartl et al.\ \cite{BaBePa21}.
		We define the adapted Wasserstein distance of $\fp X, \fp Y \in \FFP_p$ by
		\begin{equation}
			\label{eq:def_AW}
			\AW_p^p(\fp X, \fp Y) :=
			\inf_{\pi \in \cplba(\fp X,\fp Y)}
			\E_\pi \big[ d_\X^p(X, Y) \big].
		\end{equation}
		The adapted Wasserstein distance is a metric on $\FFP_p$ and we denote its induced topology by $\tau_{\AW}$.
	\item[\mylabel{it:CW}{(CW)}]
		Finally, we introduce here a new mode of convergence, the so-called topology of causal convergence $\tau_{\CW}$, which we describe below:
		A neighbourhood basis of $\fp X \in \FFP_p$ is given by
		\begin{equation}
			\big\{ \fp Y \in \FFP_p \colon \CW_p(\fp X, \fp Y) < \epsilon \big\},
		\end{equation}
		where $\epsilon > 0$.
		Hence, $\tau_\CW$ can be equivalently described by
		\begin{equation}
			\fp X^k \to \fp X \text{ in }\tau_\CW \iff \CW_p(\fp X, \fp X^k) \to 0.
		\end{equation}
\end{enumerate}

\begin{remark}
	\label{rem:causal topologies are ordered}
	It is apparent from the definitions in \eqref{eq:def_Wasserstein_distance}, \eqref{eq:def_CW}, \eqref{eq:def_symmetrized_CW} and \eqref{eq:def_AW} that
	\begin{equation}
		\label{eq:distances are ordered}
		\W_p( \fp X, \fp Y) \le \CW_p(\fp X, \fp Y) \le \SCW_p(\fp X, \fp Y) \le \AW_p(\fp X, \fp Y),
	\end{equation}
	for $\fp X,\fp Y \in \FFP_p$.
	Hence, we have $\tau_{\W} \subseteq \tau_{\CW} \subseteq \tau_{SCW} \subseteq \tau_{\AW}$.
\end{remark}

\subsection{Characterizations of the adapted weak topology}

In this subsection we formulate the main results of this paper.
The core ingredient in order to prove the main results, Theorems \ref{thm:topolgies_FFP} and \ref{thm:topologies_plain}, and also Proposition \ref{prop:topologies_weak}, is the following simple observation of topological nature.

\begin{lemma}
	\label{lem:trace topology}
	Let $(\mathcal A, \tau'), ( \mathcal A, \tau)$ be topological spaces that satisfy  the following:
	\begin{enumerate}[label = (\arabic*)]
		\item \label{it:the lemma sequential}
			$(\mathcal A,\tau')$ and $(\mathcal A,\tau)$ are sequential topological spaces.
		\item \label{it:the lemma finer topology}
			The topology $\tau$ is a least as fine as $\tau'$, that is $\tau \supseteq \tau'$.
		\item \label{it:the lemma relative compactness} 
			If $M \subset \mathcal A$ is $(\mathcal A,\tau')$-relatively compact then $M$ is $(\mathcal A, \tau)$-relatively compact.
		\item \label{it:the lemma Hausdorff}
		$(\mathcal A, \tau')$ is Hausdorff.
	\end{enumerate}
	Then we have $(\mathcal A, \tau) = (\mathcal A, \tau')$.
\end{lemma}

Note that Lemma \refeq{lem:trace topology} in combination with Theorem \refeq{thm:precompact} have Corollary \refeq{cor:equivalent_topologies} as a consequence.

Next, we provide characterizations of the adapted weak topology on $\FFP_p$.
% For a concise introduction to the variuos adapted topologies, we refer to Subsection \ref{ssec:adapted_topologies}.
% For details on $\FP$ resp.\ $\FFP$ we refer to Subsection \ref{ssec:filtered_processes}.

The equivalence of $\tau_\textrm{HK}$ and the adapted Wasserstein-topology, $\tau_\AW$, is due to \cite{BaBePa21} whereas the characterization in terms of the symmetric causal Wasserstein-topology, $\tau_\SCW$, is novel.
Moreover, we remark that the equivalence of the higher rank expected signature-topology, $\tau_{\textrm{BLO}}$ and $\tau_{\textrm{HK}}$ was already known when, for $t \in \{1,\ldots,N\}$, $\X_t = V$ and $V$ is a compact subset of a separable Banach space, see \cite[Theorem 2]{BoLiOb21}.

\begin{theorem}
	\label{thm:topolgies_FFP}
	On $\FFP_p$ we have
	\begin{equation}
		\tau_{\textrm{HK}} =
		\tau_{\SCW} =
		\tau_{\AW}.
	\end{equation}
	If $\X_t = \R^d$, $1 \le t \le N$, then these topologies also coincide with $\tau_{\textrm{BLO}}^{N-1}$, and $\tau_{\textrm{BLO}}^r = \tau_{\textrm{HK}}^r$.
\end{theorem}

When restricting to sets of processes that have a simpler information structure, e.g.\ Markov processes or processes equipped with their natural filtration, there are simpler ways to characterize the adapted weak topology.
This motivates the next definition of higher-order Markov processes where the transition probabilities are allowed to depend on more than its current state.

\begin{definition}
	\label{def:Markov_processes}
	Let $n \in \mathbb N \cup \{ \infty \}$.
	We call a process $\fp X \in \FFP_p$ $n$-th order Markovian (or $n$-th order Markov process) if, for all $1 \le t \le N$,
	\begin{equation}
		\Law(X_{t + 1} | \mathcal F_t^\fp X)
		 = \Law(X_{t + 1} | X_{1 \vee (t - n) : t})
		\quad
		\text{almost surely}.
	\end{equation}
	The set of all $n$-th order Markov processes is denoted by $\FFP_{p,n}^\textrm{Markov}$.
	Moreover, we may call $\infty$-th order Markov processes plain and write $\FFP_p^\textrm{plain} := \FFP^\textrm{Markov}_{p,\infty}$.

	We equip $\FFP_{p,n}^\textrm{Markov}$ with the initial topology $\tau^n_\textrm{Markov}$ that is given by the maps $\fp X \mapsto \Law(T^n_t(X)) \in \Pc_p(\X_{1 \vee (t - n + 1): t} \times \Pc_p(\X_{t + 1}))$ for $1 \le t \le N - 1$, where
	\begin{equation}
		T^n_t(X) :=
		\big(	
			X_{1 \vee (t - n + 1) : t}, \Law( X_{t + 1} | \mathcal F_t^\fp X)
		\big).
	\end{equation}
\end{definition}

\begin{remark}
	To illustrate Definition \ref{def:Markov_processes}, let $n = 1$.
	Clearly, $\FFP_{p,1}^{\textrm{Markov}}$ is the subset of (time-inhomogeneous) Markov processes in $\FFP_p$.
	A family of Markov processes $(\fp X^k)_{k \in \N}$ converges to a Markov process $\fp X$ w.r.t.\ $\tau_{\textrm{Markov}}^1$ if and only if, for $1 \le t \le N-1$,
	\begin{equation}
		\label{eq:rem_Markov_convergence}
		\Law\big(X^k_t, \Law(X^k_{t + 1} | X^k_t)\big) \to
		\Law\big(X_t, \Law(X_{t + 1}| X_t)\big) \quad \text{in }\Pc_p(\X_t \times \Pc_p(\X_{t + 1})).
	\end{equation}
	In particular, if there exist continuous kernels $\kappa_t \colon \X_t \to \Pc_p(\X_{t + 1})$ which satisfy $\kappa_t(X_t) = \Law(X_{t + 1} | X_t)$ almost surely, then convergence in $\tau_{\textrm{Markov}}^1$ can be characterized by the following:
	\begin{align}
		\label{eq:rem_Markov_kernel}
		\lim_{k \to \infty}
		\P^{\fp X^k} \big( \W_p( \kappa_t(X_t^k), \Law(X_{t + 1}^k | X_t^k)) \ge \epsilon \big) &= 0,
		\\
		\label{eq:rem_Markov_time_1}
		\lim_{k \to \infty} \W_p(\fp X^k, \fp X) &= 0,
	\end{align}
	for all $1 \le t \le N-1$ and $\epsilon > 0$, and some $\hat x \in \X$.
	This can be easily deduced, e.g., by using continuity of the kernels $(\kappa_t)_{t = 1}^{N-1}$ and Skorokhod's representation theorem.
\end{remark}

The next result recovers and generalizes the main result of \cite{BaBaBeEd19b}.
The novelty of the next result is two-fold:
On the one hand, the case $n = \infty$ recovers the results of \cite{BaBaBeEd19b} and additionally gives a new description in terms of $\tau^\infty_\textrm{Markov}$.
On the other hand, the case $n \in \N$ extends this result to the subset of $n$-th order Markov processes.

\begin{theorem}[All adapted topologies are equal]
	\label{thm:topologies_plain}
	Let $n, r \in \N \cup \{ \infty \}$.
	Then the trace on $\FFP_{p,n}^\textrm{Markov}$ of the topologies
	$\tau_\textrm{A}$, $\tau_\textrm{HK}^r$, $\tau_\textrm{OS}$, $\tau_\textrm{H}$, $\tau_\CW$, $\tau_\SCW$ and $\tau_\AW$ are the same.
	In particular, they all coincide with the trace of $\tau_{\textrm{Markov}}^n$.
\end{theorem}

\subsection{Characterization of the weak topology}

The line of reasoning prescribed by Lemma \ref{lem:trace topology} can be utilized outside of the framework of the adapted weak topology which is demonstrated by the proposition below.
\begin{proposition}
	\label{prop:topologies_weak}
	The $p$-Wasserstein topology on $\Pc_p(\R^d)$ can be metrized by
	\begin{equation}
		\label{eq:def_WT_metric}
		\mathcal V_{p}^p(\P,\Q) :=
		\max \big\{ V_p(\P,\Q), V_p(\Q,\P) \big\},
	\end{equation}
	where $\R^d$ is equipped with the euclidean norm $\| \cdot \|$ and
	\begin{equation}
		\label{eq:def_WT}
		V_p^p(\P,\Q) := 
		\inf_{\pi \in \cpl(\P,\Q)}
		\E_\pi \Big[ \big\| \E_\pi\big[ X - Y \big| X \big] \big\|^p \Big].
	\end{equation}
\end{proposition}

\begin{remark}
	The minimization problem depicted in \eqref{eq:def_WT} is a so-called \emph{weak optimal transport problem} \cite{GoRoSaTe17}, which is a generalization of optimal transport.
	In particular, \eqref{eq:def_WT} vanishes if and only if there exists a martingale coupling between $\P$ and $\Q$.
	For more details we refer to \cite{GoRoSaTe14,BaBePa18} and the references therein.
\end{remark}

\section{Proofs}

In order to prove the main results, we will verify the assumptions of Lemma \ref{lem:trace topology}.
By doing so, we will encounter variuous martingales which can be properly treated thanks to the next well-known fact.
We recall that a process $X = (X_t)_{t = 1}^N$ taking values in $\Pc(\X)$ is  called a measure-valued martingale with values in $\Pc(\X)$ if, for $f \in C_b(\X)$, the real-valued, bounded process $(X_t(f))_{t = 1}^N$ is a martingale.
Here, we write $p(f)$ for the integral $\int f \, dp$ when $p \in \Pc(\X)$ and $f \in C_b(\X)$.

\begin{lemma}
	\label{lem:martingale lemma}
	Let $X_1, X_2, X_3$ be measure-valued martingale taking values in $\mathcal P(\mathcal X)$ where $\X$ is a Polish space.
	If $X_1 \sim X_3$, then $X_1 = X_2 = X_3$ almost surely.
\end{lemma}

\begin{proof}
	Since there exists a countable family in $C_b(\mathcal X)$ that separates points in $\mathcal P(\mathcal X)$,
	it suffices to show that for $f \in C_b(\mathcal X)$
	\[
		X_1(f) = X_2(f) = X_3(f)\quad\text{almost surely}.
	\]
	As $X_1,X_2,X_3$ is a measure-valued martingale, we have that $Y_i := X_i(f)$ is a real-valued, bounded martingale and
	\[
		\mathbb E[ Y_1^2] = \mathbb E[ Y_2^2 ] = \mathbb E[ Y_3^2 ].
	\]
	Thus, we conclude that
	\[
		\mathbb E[ (Y_1 - Y_2)^2 ] = \mathbb E[Y_2^2 - Y_1^2] = 0.	\qedhere
	\]
\end{proof}

\subsection{Properties of \texorpdfstring{$\FFP_{p,n}^\textrm{Markov}$}{}}

First, we justify that the $n$-Markov property is preserved under equivalence.

\begin{lemma}
	Let $n \in \N \cup \{ \infty \}$, and $\fp X, \fp Y \in \FP$ with $\fp X \equiv \fp Y$.
	Then $\fp X$ is $n$-Markovian if and only if $\fp Y$ is $n$-Markovian.
\end{lemma}

\begin{proof}
	By Definition \ref{def:Markov_processes} the property of being $n$-Markovian can be deduced from observing the law of the corresponding first-order prediction process.
	Hence, we conclude by the fact that $\fp X \equiv \fp Y$ readily implies
	\[
		\Law\big( \pp^1(\fp X) \big) = \Law\big( \pp^1(\fp Y) \big).
		\qedhere
	\]
\end{proof}

\begin{lemma}
	\label{lem:CW plain}
	If $\fp X \in \FFP_p^{\text{plain}}$, $\fp Y \in \FFP_p$ and $\Law(X) = \Law(Y)$, then $\CW_p(\fp Y, \fp X) = 0$.
	In particular, if additionally $\fp Y \in \FFP_p^{\text{plain}}$, then $\fp X = \fp Y$.
\end{lemma}

\begin{proof}
	By definition of a filtered process $Y$ is adapted, therefore the coupling $\pi$, given by $(\textrm{id}_{\Omega^\fp Y}, Y)_\# \P^\fp Y$, is causal from $\fp Y$ to $\fp X := (\X, \sigma(X_{1:t})_t, \sigma(X), \Law(Y), X)$, where $X$ denotes the canonical process on $\X$.

	If $\fp Y$ is plain, c.f.\ Definition \ref{def:Markov_processes}, then
	\[
		\Law(Y | \F_t^\fp Y) = \Law(Y | Y_{1:t})
		\quad \P^\fp Y\text{-almost surely.}	
	\]
	Again, as $X$ is adapted, this translates to the following conditional independence
	\[
		X \perp_{X_{1:t}} \F_{t,0}^{\fp Y,\fp X}\quad \text{ under }\pi,
	\]
	which means that $\pi$ is bicausal and $\AW_p(\fp X, \fp Y) = 0$.
\end{proof}

\begin{corollary}
	\label{cor:Hausdorff_Markov}
	For $n, m \in \N \cup \{ \infty \}$ with $n \le m$ we have 
	$\FFP_{p,n}^\textrm{Markov} \subseteq \FFP_{p,m}^\textrm{Markov}$.
	Moreover, processes in $\FFP_{p,n}^\textrm{Markov}$ are uniquely defined by their law, that is, for $\fp X, \fp Y \in \FFP_{p,n}^\textrm{Markov}$
	\[
		\Law(X) = \Law(Y) \implies \fp X = \fp Y.
	\]
\end{corollary}

\begin{proof}
	The first claim is a direct consequence of the definition of $n$-th resp.\ $m$-th order Markov processes.
	The second claim then readily follows from Lemma \ref{lem:CW plain}.
\end{proof}

\begin{lemma}
	\label{lem:n Markov properties}
	$(\FFP_{p,n}^\text{Markov}, \tau_\text{Markov}^n)$ is a sequential Hausdorff space.
\end{lemma}

\begin{proof}
	First, we remark that, for $1 \le t \le N - 1$, the map $\fp X \mapsto \Law(T_t^n(X))$ takes values in the Polish (and therefore first countable) space
	\[
		\mathcal P_p(
			\X_{1 \vee (t - n + 1) : t} 
			\times
			\mathcal P_p(\X_{t + 1})
		),
	\]
	hence, $\tau_{n,\text{Markov}}$ is sequential.

	Next, let $\fp X, \fp Y \in \FFP_{p,n}^\text{Markov}$ with
	\[
		\Law(T^n_t(X)) = \Law(T^n_t(Y))
		\quad
		1 \le t \le N-1.
	\]
	As a direct consequence, we find
	\[
		\Law(X_{1 \vee (t - n + 1) : t + 1}) = \Law( Y_{1 \vee (t - n + 1): t + 1})\quad 1 \le t \le N-1,
	\]
	and the existence of a measurable map $f_t \colon \X_{1 \vee (t - n + 1):t} \to \mathcal P(\X_{t + 1})$ such that almost surely
	\begin{align*}
		f_t(X_{1 \vee (t - n + 1):t}) &= \Law(X_{t + 1} | X_{1 \vee (t - n + 1):t}),
		\\
		f_t(Y_{1 \vee (t - n + 1):t}) &= \Law(Y_{t + 1} | Y_{1 \vee (t - n + 1):t}).
	\end{align*}
	In particular, we have for $t = n$ that $\Law(X_{1:n + 1}) = \Law(Y_{1:n+1})$.
	We proceed to show $\Law(X) = \Law(Y)$.
	Assume that we have already shown
	\[
		\Law(X_{1:t}) = \Law(Y_{1:t})
	\]
	for some $n + 1 \le t \le N-1$.
	By the disintegration theorem and the definition of $n$-th order Markovian, we may write
	\begin{align*}
		\Law(X_{1 : t + 1}) 
		&=
		\Law(X_{1:t}) \otimes
		\Law(X_{t+1} | X_{1:t})
		\\
		&= 
		\Law(X_{1:t}) \otimes
		f_t(X_{t-n+1:t})
		\\
		&=
		\Law(Y_{1:t}) \otimes f_t(Y_{t - n + 1 :t})
		\\
		&= 
		\Law(Y_{1: t + 1}),
	\end{align*}
	where we use the notation $\mu \otimes k$ for $\mu \in \Pc(\X_{1:t})$ and a measurable kernel $k \colon \X_{1:t} \to \Pc(\X_{t + 1})$ to denote the gluing of $\mu$ with $k$, that is the probability defined by
	\[
		\mu \otimes k (A \times B) = \int_A k(x,B) \, \mu(dx)\quad A \in \mathcal B(X_{1:t}), B \in \mathcal B(X_{t+1}).
	\]
	This concludes the inductive step.

	Finally, we can apply Lemma \ref{lem:CW plain} and conclude $\fp X = \fp Y$.
\end{proof}

\begin{lemma}
	\label{lem:SCW = 0 iff AW = 0}
	Let $\fp X, \fp Y \in \FFP_p$ with $\CW_p(\fp X,\fp Y) = \CW_p(\fp Y,\fp X) = 0$, then $\fp X = \fp Y$.
	In particular, $\SCW_p(\fp X, \fp Y) = 0$ if and only if $\AW_p(\fp X, \fp Y) = 0$.
\end{lemma}

\begin{proof}
	Let $\pi \in \cpla(\fp X, \fp Y)$ and $\pi' \in \cpla(\fp Y, \fp X)$ with $X = Y$ $\pi$- and $\pi'$-almost surely.
	By \cite{BaBePa21} we may assume w.l.o.g.\ that $\F_N^\fp X = \F^\fp X$, $\F_N^\fp Y = \F^\fp Y$, and $(\Omega^\fp X, \F^\fp X)$ and $(\Omega^\fp Y, \F^\fp Y)$ are standard Borel spaces.
	This allows us to consider the conditionally independent product of $\pi$ and $\pi'$ denoted by $\hat \pi := \pi \dot \otimes \pi' \in \cpl(\fp X, \fp Y, \fp X)$, see Definition \ref{def:CIprod}.
	Here, $\cpl(\fp X, \fp Y, \fp X)$ denotes the set of coupling with marginals $\P^\fp X$, $\P^\fp Y$ and $\P^\fp X$.
	We write $\tilde{\fp X}$ and $\tilde X$ for the second $\fp X$-coordinate in order to distinguish them.
	By induction we show that $\hat \pi$-almost surely
	\begin{equation}
		\label{eq:k th order pp coincides}
		\pp^k(\fp X) = \pp^k(\fp Y) = \pp^k(\tilde{\fp X}).
	\end{equation}
	for all $k \in \N \cup \{ 0 \}$.
	Since we know that $X = Y = \tilde X$ $\hat \pi$-almost surely,
	we have verified \eqref{eq:k th order pp coincides} for $k = 0$.
	Assume that \eqref{eq:k th order pp coincides} holds for some $k$.
	By causality of $\pi'$ and Lemma \ref{lem:CI.causal} we find, for $1 \le t \le N$,
	\begin{align}
		\label{eq:conditional independencies of product}
		\F_{0,N,0}^{\fp X,\fp Y, \tilde{\fp X}}
		\perp_{\F_{0,t,0}^{\fp X,\fp Y,\tilde{\fp X}}}		
		\F_{0,0,t}^{\fp X,\fp Y, \tilde{\fp X}},
		\quad
		\F_{N,0,0}^{\fp X,\fp Y, \tilde{\fp X}}
		\perp_{\F_{t,0,0}^{\fp X,\fp Y,\tilde{\fp X}}}		
		\F_{0,t,t}^{\fp X,\fp Y, \tilde{\fp X}},
	\end{align}
	where naturally extend the notation introduced in Definition \ref{def:causality} in order to write products of multiple $\sigma$-algebras.
	Since $\pp^k(\fp X)$ is $\F_{N,0,0}^{\fp X,\fp Y, \tilde{\fp X}}$-measurable and $\pp^k(\fp Y)$ is $\F_{N,N,0}^{\fp X,\fp Y, \tilde{\fp X}}$-measurable, we obtain by combining \eqref{eq:k th order pp coincides}, \eqref{eq:conditional independencies of product}, and the tower property
	\begin{align*}
		\pp^{k+1}_t(\tilde{\fp X}) &= 
		\Law \big( \pp^k(\tilde{\fp X})| \F_{0,0,t}^{\fp X, \fp Y, \tilde{\fp X}}\big)
		\\
		&=
		\E \Big[ \Law \big( \pp^k(\fp Y)| \F_{0,t,t}^{\fp X, \fp Y, \tilde{\fp X}}\big) | \F_{0,0,t}^{\fp X, \fp Y, \tilde{\fp X}}\Big]
		\\
		&=
		\E \Big[ 
			\pp^{k + 1}_t(\fp Y)
			| \F_{0,0,t}^{\fp X, \fp Y, \tilde{\fp X}}
		\Big],
	\end{align*}
	and similarly,
	\begin{align*}
		\pp_t^{k+1}(\fp Y) 
		&=
		\E \Big[ \pp_t^{k + 1}(\fp X) | \F_{0,t,t}^{\fp X,\fp Y,\tilde{\fp X}} \Big].
	\end{align*}
	Hence, the triplet $(\pp_t^{k + 1}(\tilde{\fp X}), \pp_t^{k + 1}(\fp Y)), \pp_t^{k + 1}(\fp X))$ satisfies the assumptions of Lemma \ref{lem:martingale lemma}, which concludes the inductive step.
	In particular, we have shown that $\pp(\fp X) \sim \pp(\fp Y)$, whence $\fp X = \fp Y$ by \cite[Theorem 4.11]{BaBePa21}.
\end{proof}

\begin{proposition}
	\label{prop:rel_comp_Markov}
	Let $n \in \N \cup \{ \infty \}$ and $M\subseteq \FFP_{p,n}^\text{Markov}$ be $\tau_{\text{Markov}}^n$-relatively compact.
	Then $M$ is relatively compact in $(M,\tau_\AW)$.
\end{proposition}

\begin{proof}
	Let $(\fp X^k)_{k \in \N}$ be a sequence in $\FFP_{p,n}^\textrm{Markov}$ $\tau_{\textrm{Markov}}^n$-converging to $\fp X \in \FFP_{p,n}^\textrm{Markov}$.
	First, we convince ourselves that $(\Law(X^k))_{k \in \N}$ converges to $\Law(X)$:
	Assume that we have already shown that $\Law(X_{1:t}^k) \to \Law(X_{1:t})$ for some $1 \le t \le N-1$.
	The conditionally independent product $\dot\otimes$, see \cite[Definition 2.8]{Ed19}, allows us to rewrite
	\[
		\Law(X_{1:t}, \Law(X_{t + 1} | X_{1:t})) = \Law(X_{1:t}) \dot \otimes \Law(T_t^n(X)).
	\]
	By \cite[Theorem 4.1]{Ed19}, that is in our context continuity of $\dot\otimes$ at $(\Law(X_{1:t}), \Law(T_t^n(X)))$, we obtain that $\Law(X_{1:t+1}^k) \to \Law(X_{1:t+1})$.

	Hence, $(\Law(X^k))_{k \in \N}$ is convergent and therefore tight. 
	Thus, there exists by Theorem \refeq{thm:precompact} a subsequence of $(\fp X^k)_{k \in \N}$ converging in $\tau_{\AW}$ to some $\fp Y \in \FFP_p$.
	Due to $\tau_\AW$-continuity, we get
	\[
		\Law( T^n_t(X))	= \lim_{j \to \infty} \Law(T_t^n(X^{k_j})) = \Law( T^n_t(Y)).
	\]
	Hence, there exist measurable maps $f_t \colon \X_{1 \vee (t - n + 1):t} \to \Pc(\X_{t+1})$ with the property
	\[
		f_t(Y_{1 \vee (t - n + 1):t}) = \Law(Y_{t+1} | \F_t^\fp Y)
		\quad
		\text{almost surely}.
	\]
	In other words, $\fp Y \in \Lambda_{n,\text{Markov}}$.
	Therefore the sequence $(\fp X^k)_{k \in \N}$ is also relatively compact in $(\FFP_{p,n}^\textrm{Markov}, \tau_{\AW})$, which concludes the proof.
\end{proof}

\begin{proposition}
	\label{prop:rel_comp_CW}
	Let $M \subseteq \FFP_p^\textrm{plain}$ be relatively compact in $(\FFP_p^\textrm{plain}, \tau_\CW)$.
	Then $M$ is relatively compact in $(\FFP_p^\textrm{plain}, \tau_\AW)$.
\end{proposition}

\begin{proof}
	Let $M$ be relatively compact in $(\FFP_p^\textrm{plain}, \tau_\CW)$.
	Since $\tau_\W \subseteq \tau_\CW$, there exists by Theorem \ref{thm:precompact} a $\tau_{\AW}$-convergent subsequence with limit $\fp Y$ for some $\fp Y \in \FFP_p$.
	Since $\CW_p$ is by Lemma \ref{lem:CW 1-Lipschitz} (1-Lipschitz) continuous w.r.t.\ $\AW_p$, we find
	\[
		\CW_p(\fp X, \fp Y) = \lim_j \CW_p(\fp X, \fp X^{k_j}) = 0,
	\]
	and conclude with Lemma \ref{lem:CW plain} that $\fp Y (= \fp X) \in \FFP_p^\textrm{plain}$.
\end{proof}

\subsection{Causal gluing}
\label{ssec:causal.gluing}

This section is devoted to develop auxiliary results concerning the composition of causal couplings with matching intermediary marginal.
We recall that due to \cite{BaBePa21} we can always assume w.l.o.g.\ that all spaces under consideration are standard Borel.
Therefore, we assume for the rest of the section that we have chosen representatives of $\fp X, \fp Y, \fp Z \in \FFP$ such that
\[
	\big( \Omega^\fp X, \F_N^\fp X \big), \quad
	\big( \Omega^\fp Y, \F_N^\fp Y \big), \quad
	\big( \Omega^\fp Z, \F_N^\fp Z \big),
\]
are standard Borel.
\begin{definition} \label{def:CIprod}
	Let $\gamma \in \cpl(\fp X, \fp Y)$ and $\eta \in \cpl(\fp Y, \fp Z)$.We define the \emph{conditionally independent product} of $\gamma$ and $\eta$ as the probability on $\big(\Omega^\fp X \times \Omega^\fp Y \times \Omega^\fp Z, \F^{\fp X,\fp Y, \fp Z}_{N,N,N} \big)$ satisfying for any $U$, bounded and $\F^{\fp X,\fp Y, \fp Z}_{N,N,N}$-measurable, that
	\begin{equation}
		\label{eq:def_CIprod}
		\int U \, d\gamma \dot \otimes \eta =
		\int \int U(\omega^\fp X, \omega^\fp Y, \omega) \, \eta_{\omega^\fp Y} (d \omega^\fp Z) \, \gamma(d\omega^\fp X, d\omega^\fp Y),
	\end{equation}
	where $\eta_{\omega^\fp Y}$ is a disintegration kernel of $\eta$ w.r.t.\ the projection on $\Omega^\fp Y$.
	Due to symmetry reasons, we have
	\begin{equation}
		\label{eq:cond.indep.prod.both.disintegrations}
		\int U \, d \gamma \dot \otimes \eta
		=  \iint U(\omega^\mathbf{X},\omega^\mathbf{Y},\omega^\mathbf{Z}) \,(\eta_{\omega^\mathbf{Y}} \otimes  \gamma_{\omega^\mathbf{Y}})(d\omega^\mathbf{X},d\omega^\mathbf{Z})\,\P^\fp Y(d\omega^\mathbf{Y}).
	\end{equation}
\end{definition}

The term \eqref{eq:cond.indep.prod.both.disintegrations} clarifies the naming of $\gamma \dot \otimes \eta $ as the conditional independent product: conditionally on $\omega^\mathbf{Y}$ the knowledge of $\omega^\mathbf{X}$ does not affect $\omega^\mathbf{Z}$ and vice versa.
This suggests the following probabilistic formulation.

\begin{lemma} \label{lem:CI.product}
	Let $\gamma \in \cpl(\fp X,\fp Y)$ and $\eta \in \cpl(\fp Y, \fp Z)$.
	We have under $\gamma \dot \otimes \eta$
	\[
		\F^{\fp X, \fp Y, \fp Z}_{N,N,0} \perp_{\F^{\fp X,\fp Y,\fp Z}_{0,N,0}} \F_{0,N,N}^{\fp X, \fp Y, \fp Z}.
	\]
	In particular, if $\mathcal G$ is a $\sigma$-algebra with $\F^{\fp X,\fp Y,\fp Z}_{0,N,0}\subseteq\mathcal G \subseteq \F^{\fp X,\fp Y, \fp Z}_{N,N,0}$, we have under $\gamma \dot\otimes \eta$
	\[
		\F^{\fp X, \fp Y, \fp Z}_{N,N,0} \perp_\mathcal G \F_{0,N,N}^{\fp X, \fp Y, \fp Z}.
	\]
\end{lemma}

\begin{proof}
	Let $U$, $V$, $W$ be bounded and $\F^{\fp X, \fp Y, \fp Z}_{N,0,0}$-measurable, $\F^{\fp X, \fp Y, \fp Z}_{0,N,0}$-measurable, and $\F^{\fp X, \fp Y, \fp Z}_{0,0,N}$-measurable, respectively.
	Write $\hat W$ for the bounded, $\F^{\fp X, \fp Y, \fp Z}_{0,N,0}$-measurable random variable given by $\int W(\omega^\fp Z) \, \eta_{\omega^\fp Y}(d \omega^\fp Z)$.
	By Definition \ref{def:CIprod} and the tower property we get
	\[
		\E_{\gamma \dot \otimes \eta} \left[ U V W \right] = \E_{\gamma \dot \otimes \eta} \left[ U V \hat W \right]
		= \E_{\gamma \dot \otimes \eta} \left[\E_{\gamma \dot \otimes \eta}\left[U | F_{0,N,0}^{\fp X, \fp Y, \fp Z} \right]  V  \hat W \right].
	\]
	Since $\hat W$ coincides with $\E_{\gamma \dot \otimes \eta}[W | \F^{\fp X, \fp Y, \fp Z}_{0,N,0}]$ and $V$ was arbitrary, we derive
	\[
		\E_{\gamma \dot \otimes \eta}\left[UW | \F^{\fp X, \fp Y, \fp Z}_{0,N,0} \right]	= \E_{\gamma \dot \otimes \eta}\left[ U | \F^{\fp X, \fp Y, \fp Z}_{0,N,0} \right] \E_{\gamma \dot \otimes \eta}\left[ W | \F^{\fp X, \fp Y, \fp Z}_{0,N,0} \right],
	\]
	which shows the first statement.

	The second statement is a consequence of applying \cite[Proposition 5.8]{Ka97} to the previously shown.
	\end{proof}

\begin{lemma} \label{lem:CI.causal}
	Let $\gamma \in \cpl(\fp X,\fp Y)$ and $\eta \in \cpla(\fp Y, \fp Z)$.
	We have, for $1 \le t \le N$,
	\begin{enumerate}[label=(\arabic*)]
		\item \label{it:CI.causal1} under $\gamma \dot \otimes \eta\colon$ $\F^{\fp X, \fp Y, \fp Z}_{N,N,0} \perp_{F^{\fp X, \fp Y, \fp Z}_{t,t,0}} \F_{0,t,t}^{\fp X, \fp Y, \fp Z}$;
	\end{enumerate}
	if furthermore $\gamma \in \cpla(\fp X, \fp Y)$, then we have
	\begin{enumerate}[resume,label=(\arabic*)]
		\item \label{it:CI.causal2} under $\gamma \dot \otimes \eta\colon \F^{\fp X, \fp Y, \fp Z}_{N,0,0}\perp_{F^{\fp X, \fp Y, \fp Z}_{t,0,0}}\F_{0,t,t}^{\fp X, \fp Y, \fp Z}$.
	\end{enumerate}
\end{lemma}

\begin{proof}
	To show item \ref{it:CI.causal1}, let $W$ be bounded and $\F^{\fp X,\fp Y,\fp Z}_{0,t,t}$ measurable.
	We obtain from Lemma \ref{lem:CI.product} the first equality in
	\begin{equation}
		\label{eq:CI.causal1}
		\E_{\gamma \dot \otimes \eta} \left[ W | \F^{\fp X, \fp Y, \fp Z}_{N,N,0} \right] = \E_{\gamma \dot \otimes \eta} \left[ W | \F^{\fp X, \fp Y, \fp Z}_{0,N,0}  \right] = \E_{\gamma \dot \otimes \eta} \left[ W | \F^{\fp X, \fp Y, \fp Z}_{0,t,0}  \right],
	\end{equation}
	whereas the second stems from causality of $\eta$.
	Here this causality yields under $\gamma \dot \otimes \eta$ that, conditionally on $\F^{\fp X, \fp Y, \fp Z}_{0,t,0}$, $\F^{\fp X,\fp Y,\fp Z}_{0,N,0}$ is independent of $\F^{\fp X, \fp Y, \fp Z}_{0,t,t}$.
	Since the last term in \eqref{eq:CI.causal1} is $\F^{\fp X, \fp Y, \fp Z}_{t,t,0}$-measurable, the tower property yields item \ref{it:CI.causal1}.

	To establish item \ref{it:CI.causal2}, let $W$ be as above.
	Note that causality of $\gamma$ provides under $\gamma \dot \otimes \eta$ that, conditionally on $\F^{\fp X,\fp Y,\fp Z}_{t,0,0}$, $\F^{\fp X, \fp Y, \fp Z}_{N,0,0}$ is independent of $\F^{\fp X, \fp Y, \fp Z}_{t,t,0}$.
	Using that in addition to item \ref{it:CI.causal1} and the tower property, we conclude
	\begin{align*}
		\E_{\gamma \dot \otimes \eta}\left[ W | \F^{\fp X, \fp Y, \fp Z}_{N,0,0}\right] &= \E_{\gamma \dot \otimes \eta}\left[ \E_{\gamma \dot \otimes \eta}\left[ W | \F^{\fp X, \fp Y, \fp Z}_{t,t,0}\right]  | \F^{\fp X, \fp Y, \fp Z}_{N,0,0}\right] 
	\\ &= \E_{\gamma \dot \otimes \eta}\left[ \E_{\gamma \dot \otimes \eta}\left[ W | \F^{\fp X, \fp Y, \fp Z}_{t,t,0}\right]  | \F^{\fp X, \fp Y, \fp Z}_{t,0,0}\right] = \E_{\gamma \dot \otimes \eta}\left[ W | \F^{\fp X, \fp Y, \fp Z}_{t,0,0}\right].\qedhere
	\end{align*}
\end{proof}

\begin{corollary}
	\label{cor:gluing}
	Let $\gamma \in \cpla(\fp X,\fp Y)$ and $\eta \in \cpla(\fp Y, \fp Z)$.
	We write $\proj_{\Omega^\fp X \times \Omega^\fp Z}$ for the projection onto $\Omega^\fp X \times \Omega^\fp Z$, then $(\proj_{\Omega^\fp X \times \Omega^\fp Z})_\# \gamma \dot \otimes \eta \in \cpla(\fp X, \fp Z)$.
\end{corollary}

\begin{proof}
	This result is a direct consequence of item \ref{it:CI.causal2} of Lemma \ref{lem:CI.causal}.
\end{proof}

\begin{lemma}
	\label{lem:CW 1-Lipschitz}
	Let $\fp X \in \FFP_p$.
	The map
	\begin{equation}
		\FFP_p \ni \fp Y \mapsto \CW_p(\fp X,\fp Y)
	\end{equation}
	is 1-Lipschitz w.r.t.\ $\SCW_p$.
\end{lemma}

\begin{proof}
	Let $\pi \in \cpla(\fp X,\fp Y)$ and $\pi' \in \cpla(\fp Y, \fp Z)$, then $(\proj_{\Omega^\fp X \times \Omega^\fp Z} \pi \dot\otimes \pi' )_\# \in \cpla(\fp X, \fp Z)$ by Corollary \ref{cor:gluing}.
	Hence, we compute
	\begin{align*}
		\CW_p(\fp X, \fp Z) &\le 
		\Big( \E_{\pi \dot\otimes \pi'} \big[ d^p_\X(X,Z) \big] \Big)^\frac1p
		\\
		&\le
		\Big( \E_\pi\big[ d_\X^p(X,Y) \big] \Big)^\frac1p + \Big( \E_{\pi'} \big[ d_\X^p(Y,Z)\big]\Big)^\frac1p,
	\end{align*}
	and conclude
	\[
		\big| \CW_p(\fp X,\fp Z) - \CW_p(\fp X, \fp Y) \big| \le \SCW_p(\fp Y,\fp Z).\qedhere	
	\]
\end{proof}

\subsection{Postponed proofs of Section \ref{sec:introduction}}

\begin{proof}[Proof of Lemma \ref{lem:trace topology}]
	Due to \ref{it:the lemma finer topology} it remains to show that convergence in $(\mathcal A,\tau')$ implies convergence in $(\mathcal A, \tau)$.
	To this end, let $(y^k)_{k \in \N}$ be a sequence in $(\mathcal A,\tau')$ converging to $y$.
	By \ref{it:the lemma relative compactness} we find a subsequence $(y^{k_j})_{j \in \mathbb N}$ that converges in $(\mathcal A,\tau)$ to some element $z$.
	Again, by \ref{it:the lemma finer topology} we have that $(y^{k_j})_{j \in \mathbb N}$ also converges in $(\mathcal A,\tau')$ to $z$, which yields by \ref{it:the lemma Hausdorff} that $y = z$.
	Therefore, $y$ is the only $(\mathcal A,\tau)$-accumulation point of $(y^k)_{k \in \N}$, from where we conclude that $(y^k)_{k \in \N}$ has to converge to $y$ in $(\mathcal A,\tau)$.
\end{proof}

\begin{proof}[Proof of Theorem \ref{thm:topolgies_FFP}]
	It is evident from \cite[Theorem 3.10]{BaBePa21}, \cite[Lemma 4.7]{BaBePa21} and \cite[Lemma 4.10]{BaBePa21} that $\tau_\AW$ and $\tau_\textrm{HK}$ coincide.

	Using the notation of Lemma \ref{lem:trace topology}, we let $(\mathcal B, \tau) := (\FFP_p,\tau_\AW)$ and $\mathcal A = \mathcal B$.
	By Remark \ref{rem:causal topologies are ordered}, resp.\ \cite[Proposition 6]{BoLiOb21} we have for $\tau' \in \{ \tau_{\SCW}, \tau_{\textrm{BLO}}^{N-1} \}$ that $\tau' \subseteq \tau$.
	By Lemma \ref{lem:SCW = 0 iff AW = 0}, resp.\ \cite[Theorem 4]{BoLiOb21} we find that $\tau'$ is Hausdorff.
	Moreover, we obtain $\tau_\W \subseteq \tau'$ from Remark \ref{rem:causal topologies are ordered} resp.\ \cite[Proposition 8]{BoLiOb21}, where $\tau_\W$ is the topology of $p$-Wasserstein convergence of the laws.
	Since $\tau_\W$ and $\tau$ have the same relatively compact sets by Theorem \ref{thm:precompact}, we conclude the same for $\tau'$.
	Hence, all assumptions of Lemma \ref{lem:trace topology} are met which yields the first two assertions of the theorem.

	The last assertion of the theorem follows mutatis mutandis.
\end{proof}

\begin{proof}[Proof of Theorem \ref{thm:topologies_plain}]
	Let $\mathcal A := \FFP_{p,n}^\textrm{Markov}$, $\mathcal B := \FFP_p$, and $\tau = \tau_\AW$.

	It is evident (either by construction, from Theorem \ref{thm:topolgies_FFP}, or from \cite[Lemma 7.5]{BaBaBeEd19b}) that $\tau_\textrm{Markov}^n$ is coarser than $\tau_\textrm{H}$, $\tau_\textrm{A}$, $\tau_\textrm{HK}^r$, $\tau_\textrm{OS}$, $\tau_\textrm{AW}$ and $\tau_\textrm{SCW}$.
	Similarly, we have that all of these topologies are coarser than $\tau_\AW$.
	We remark that $\tau_\AW \supseteq \tau_\textrm{OS}$ can be seen due to the fact that the map which maps $\fp X \in \FFP_p$ to its Snell envelope is $\tau_\AW$-continuous.

	Thus, it suffices to show that $(\mathcal A, \tau') = (\mathcal A,\tau)$ for $\tau' \in \{ \tau_\textrm{Markov}^n, \tau_\CW \}$.
	We proceed by verifying the assumptions in Lemma \ref{lem:trace topology}:
	Item \ref{it:the lemma sequential} follows from Lemma \ref{lem:n Markov properties} resp.\ is evident by construction.
	Item \ref{it:the lemma finer topology} is satisfied, since it is easy to see that $\tau_\textrm{H} \subseteq \tau_\textrm{HK}^1 \subseteq \tau_\textrm{HK} = \tau_\textrm{AW}$ (where the last equality is due to Theorem \ref{thm:topolgies_FFP}) resp.\ by Remark \ref{rem:causal topologies are ordered}.
	Item \ref{it:the lemma relative compactness} is proven in Proposition \ref{prop:rel_comp_Markov} resp.\ Proposition \ref{prop:rel_comp_CW}.
	Finally, item \ref{it:the lemma Hausdorff} is due to Corollary \ref{cor:Hausdorff_Markov} resp.\ Lemma \ref{lem:CW plain}.
\end{proof}

\begin{proof}[Proof of Proposition \ref{prop:topologies_weak}]
	Let $\mathcal A = \mathcal B = \Pc_p(\R^d)$ and $\tau = \tau_\W$.
	It is straightforward to check that $\mathcal V_p$ is a pseudometric and $\mathcal V_p \le \W_p$.
	Moreover, as a simple consequence of Lemma \ref{lem:martingale lemma} we find that $\mathcal V_p$ separates points:
	If $\mathcal V_p(\P,\Q) = 0$ then there exist martingale couplings $\pi \in \cpl(\P,\Q)$ and $\tilde \pi \in \cpl(\Q,\P)$.
	Let $X = (X_t)_{t =  1}^3$ be a Markov process with $(X_1,X_2) \sim \pi$ and $(X_2,X_3) \sim \pi'$.
	Therefore, $X$ is a martingale and by Lemma \ref{lem:martingale lemma} $X_1 \sim X_2$, that is $\P = \Q$ and $\mathcal V_p$ is a metric on $\Pc_p(\R^d)$.
	We write $\tau_\mathcal V$ for the topology induced by $\mathcal V_p$ and get $\tau_\mathcal V \subseteq \tau_\W$.
	It remains to verify Item \ref{it:the lemma relative compactness} of Lemma \ref{lem:trace topology}.
	
	To this end, let $(\P^k)_{k \in \N}$ converge to $\P$ in $\tau_\mathcal V$  and we want to show $\W_p$-relative compactness of the sequence.
	By \cite[Lemma 6.1]{BaBePa18}, we have
	\begin{equation}
		\label{eq:V_p = W_p proj}		
		V_p(\P^k,\P) = 
		\inf_{\Q \le_{\textrm{cx}} \P}
		\W_p(\P^k, \Q),
	\end{equation}
	where $\le_{\textrm{cx}}$ denotes the convex order on $\P_1(\R^d)$.
	Recall that, for $\mu,\nu \in \P_1(\R^d)$, $\mu \le_\textrm{cx} \nu$ if and only if $\int f \, d\mu \le \int f \, d\nu$ for all $f \colon \R^d \to \R^d$ convex.
	Due to compactness of closed balls (of finite radius) in $\R^d$ it is easy to see, for example by proper application of the De la Vall\'ee-Poussin theorem for uniform integrability and \cite[Definition 6.8]{Vi09}, that the set $\{ \Q \le_\textrm{cx} \P \}$ is $\W_p$-compact in $\P_p(\R^d)$, hence, we find by standard arguments the existence of $\Q^k \le_\textrm{cx} \P$ attaining \eqref{eq:V_p = W_p proj}.
	Consequentially,
	\[
		\lim_{k \to \infty} V_p(\P,\P^k) = \lim_{k \to \infty} \W_p(\P^k,\Q^k) = 0,
	\]
	which in particular yields $\W_p$-relative compactness of $\{ \P^k \colon k \in \N\}$.
\end{proof}

%-----------------------   bibliography ---------------------------------------

\bibliographystyle{abbrv}
%\bibliography{bib.bib}
\bibliography{../MBjointbib/joint_biblio}
\end{document}